\documentclass[12pt, twoside]{article}
\usepackage{amsmath,amsthm,amssymb}
\usepackage{times}
\usepackage{enumerate}
\usepackage{bm}
\usepackage[all]{xy}
\usepackage{mathrsfs}
\usepackage{amscd}
\usepackage{color}
\def\titlerunning#1{\gdef\titrun{#1}}
\makeatletter
\def\author#1{\gdef\autrun{\def\and{\unskip, }#1}\gdef\@author{#1}}
\def\address#1{{\def\and{\\\hspace*{18pt}}\renewcommand{\thefootnote}{}%
		\footnote {#1}}%
	\markboth{\autrun}{\titrun}}
\makeatother
\def\email#1{e-mail: #1}
\def\subjclass#1{{\renewcommand{\thefootnote}{}%
		\footnote{\emph{Mathematics Subject Classification (2010):} #1}}}
\def\keywords#1{\par\medskip
	\noindent\textbf{Keywords.} #1}

\newtheorem{theorem}{Theorem}[section]
\newtheorem{corollary}[theorem]{Corollary}

\newtheorem{proposition}[theorem]{Proposition}
\theoremstyle{definition}
\newtheorem{definition}[theorem]{Definition}
\newtheorem{remark}[theorem]{Remark}

\theoremstyle{example}

\numberwithin{equation}{section}

\xyoption{all}

\def \N {\mathbb{N}}
\def \C {\mathbb{C}}

\def \De {\Delta}

\def\Om{\Omega}

\def\na {\nabla}
\def\Ga{\Gamma}

\begin{document}
\baselineskip=17pt

\titlerunning{A compactness theorem for stable flat $SL(2,\mathbb{C})$ connections on $3$-folds}
\title{A compactness theorem for stable flat $SL(2,\mathbb{C})$ connections on $3$-folds}

\author{Teng Huang}

\date{}

\maketitle

\address{Teng Huang: School of Mathematical Sciences, University of Science and Technology of China; CAS Key Laboratory of Wu Wen-Tsun Mathematics,  University of Science and Technology of China, Hefei, Anhui, 230026, People’s Republic of China; \email{htmath@ustc.edu.cn;htustc@gmail.com}}

\begin{abstract}
Let $Y$ be a closed $3$-manifold such that all flat $SU(2)$-connections on $Y$ are $non$-$degenerate$. In this article, we prove a Uhlenbeck-type compactness theorem on $Y$ for stable flat $SL(2,\mathbb{C})$ connections satisfying an $L^{2}$-bound for the real curvature. Combining the compactness theorem and a previous result in \cite{Huang}, we prove that the moduli space of the stable flat $SL(2,\mathbb{C})$ connections is disconnected under certain technical assumptions.
\end{abstract}
\section{Introduction}
Let $X$ be an oriented, closed, smooth $n$-dimensional manifold with smooth Riemannian metric $g$, and let $P$ be a principal $G$-bundle over $X$ with $G$ being a compact Lie group. We denote by $\mathcal{A}_{P}$ the set of all connections on $P$, and by $\Om^{k}(X,\mathfrak{g}_{P})$ the set of  $\mathfrak{g}_{P}$-valuled $k$-forms, where  $\mathfrak{g}_{P}$ is the adjoint bundle of $P$. Suppose that $A$ is a connection on $P$ and its curvature is denote by $F_{A}\in\Om^{2}(X,\mathfrak{g}_{P})$. For any connection $A$ on $P$ we have the covariant exterior derivatives $d_{A}:\Om^{k}(X, \mathfrak{g}_{P})\rightarrow\Om^{k+1}(X,\mathfrak{g}_{P})$. The curvature $$\mathcal{F}_{\mathcal{A}}=F_{A}-\frac{1}{2}[\phi\wedge\phi]+{\rm{i}}d_{A}\phi$$
of the complex connection $\mathcal{A}:=A+\rm{i}\phi$ is a $2$-form with values in $P\times_{G}(\mathfrak{g}_{P}^{\C})$. We say that $\mathcal{A}=A+\rm{i}\phi$ is a complex flat connection with the moment map condition, if the pair $(A,\phi)$ satisfies,
\begin{equation}\label{I1}
F_{A}-\phi\wedge\phi=0,\ d_{A}\phi=d_{A}^{\ast}\phi=0.
\end{equation}
The system of the pairs $(A,\phi)$ is elliptic \cite{Gagliardo/Uhlenbeck:2012} For convenience, we call the solutions of the elliptic system as stable flat connections, see \cite{Corlette}. These equations are not only invariant under the actions of real gauge group $\mathcal{G}_{P}=C^{\infty}(P\times_{G}G)$, but also invariant under the actions of complex gauge group $\mathcal{G}^{\C}_{P}:=C^{\infty}(P\times_{G}G_{\C})$. The solution of stable flat connections on compact Riemannian surface $\Sigma$ is also a solution of Hitchin's equation \cite{Hitchin}. The moduli space of the solutions of Hitchin's equations which satisfying $\int_{\Sigma}|\phi|^{2}\leq K$ is compact, see \cite[Theorem 4.1]{Gagliardo/Uhlenbeck:2012}. Following \cite[Proposition 2.2.3]{Donaldson/Kronheimer} or \cite[Proposition 1.2.6]{Kobayashi},  we know that the gauge-equivalence classes of flat connections over a connected manifold, $X$, are in one-to-one correspondence with the conjugacy classes of representations $\pi_{1}(X)\rightarrow G$. 

The Uhlenbeck compactness theorem \cite{Uhlenbeck1982,Wehrheim} on the moduli space of the connections with $L^{p}$-bounds on curvature is one of the most fundamental theorems in the analytical aspect of the gauge theory. In \cite{Taubes2013}, Taubes studied the Uhlenbeck style compactness problem for $SL(2,\C)$ connections, including solutions to the above equations, on three-, four-manifold \cite{Taubes2013,Taubes2014.07,Taubes2014}.

We denote by  $$\mathcal{M}(P,g):=\{(A,\phi)\in\mathcal{A}_{P}\times\Om^{1}(\mathfrak{g}_{P})\mid \mathcal{F}_{\mathcal{A}}=0,\ d^{\ast}_{A}\phi=0\}/\mathcal{G}_{P},$$
the moduli space of the stable flat $SL(2,\mathbb{C})$ connections. In particular, the moduli space of gauge-equivalence classes $[\Ga]$ of flat connections $\Ga$ on $P$, $$M(P,g):=\{\Ga\in\mathcal{A}_{P}: F_{\Ga}=0\}/\mathcal{G}_{P}$$
can be embedded into $\mathcal{M}(P,g)$ via the map $A\mapsto(A,0)$. The Uhlenbeck compactness theorem \cite{Uhlenbeck1982} shows that the moduli space $M(P,g)$ is compact.

One can see that the pair $(A,\phi)$ has the a priori estimate (see \cite{Gagliardo/Uhlenbeck:2012})
\begin{equation}\label{E1.2}
\int_{X}|\na_{A}\phi|^{2}+|F_{A}|^{2}\leq c_{0}\int_{X}|\phi|^{2}, 
\end{equation} 
where $c_{0}$ is a positive constant dependent on the metric $g$. Then the Uhlenbeck compactness theorem implies that the moduli space of solutions of stable flat $SL(2,\mathbb{C})$ connections satisfying (\ref{I1}) with $\int_{X}|\phi|^{2}\leq K$ is compact, for every given positive constant $K$. On the other hand, there are examples of sequences of solutions $(A_{i},\phi_{i})$ to (\ref{I1}) such that $\|\phi_{i}\|_{L^{2}(X)}$ diverges to infinity, therefore the moduli space of solutions to (\ref{I1}) is not always compact. An interesting question to ask if following:\\
let $K$ be a positive constant number, and consider the subset of $\mathcal{M}(P,g)$ consisting of $(A,\phi)$ such that $\|F_{A}\|_{L^{2}(X)}\leq K$ is this subset alawys compact?\\
In this article, we consider the case for  the stable flat $SL(2,\mathbb{C})$ connections on  a closed, smooth, oriented three-manifold $Y$. We will give a positive answer for this question if $Y,G,P$ satisfy certain conditions.

We denote by 
$$\De_{\Ga}:=d_{\Ga}^{\ast}d_{\Ga}+d_{\Ga}d_{\Ga}^{\ast}$$
the self-dual operator with respect to a flat connection $\Ga$. We recall the definition of \textit{non-degenerate} flat connections as follows, see \cite[ Definition 2.4]{Donaldson}.
\begin{definition}\label{D3.6}
	Let $G$ be a compact Lie group, $P$ be a $G$-bundle over a closed, smooth manifold $X$ of dimension $n\geq2$ and endowed with a smooth Riemannian metric $g$. The flat connection $\Ga$ called $non$-$degenerate$ if $$\ker{\De_{\Ga}}|_{\Om^{1}(X,\mathfrak{g}_{P})}=\{0\}.$$
\end{definition}
The main observation of this article can be stated as follows.
\begin{theorem}\label{T1.1}(A compactness theorem for stable flat $SL(2,\mathbb{C})$ connections with bounded real curvatures). Let $Y$ be a closed, oriented, smooth Riemannian  three-manifold, $P$ be a principal  $SU(2)$ or $SO(3)$-bundle over $Y$. Let $\{(A_{i},\phi_{i})\}_{i\in\mathbb{N}}$ be a sequence of $C^{\infty}$-solutions of Equations (\ref{I1}). Suppose that all flat connections on the principal bundle $P$ are $non$-$degenerate$. If the $L^{2}$-norms of the curvatures $\|F_{A_{i}}\|_{L^{2}(Y)}$ are bounded, then there is a subsequence of $\Xi\subset\mathbb{N}$ and a sequence of gauge transformations $\{u_{i}\}_{i\in\Xi}$ such that $\{(u^{\ast}_{i}(A_{i}),u_{i}^{\ast}(\phi_{i}))\}_{i\in\Xi}$ converges to a pair $(A_{\infty},\phi_{\infty})$ obeying Equations (\ref{I1}) on $P$ in the $C^{\infty}$-topology. In particular, the moduli space of solutions of stable flat $SL(2,\mathbb{C})$ connections which obeys $\int_{Y}|F_{A}|^{2}\leq K$ is compact, for every positive constant $K$.
\end{theorem}
\begin{remark}
	Taubes \cite{Taubes2013} considered a sequence of complex connections  $\mathcal{A}_{i}:=A_{i}+{\rm{i}}\phi_{i}$ such that the $L^{2}$-norms of $F_{\mathcal{A}_{i}}$ are bounded (in this article, the complex curvature is just zero). There are two possible cases: (1) if $\|\phi_{i}\|_{L^{2}(X)}$ has a bounded sequence, then Taubes proves $(u^{\ast}_{i}(A_{i}),u_{i}^{\ast}(\phi_{i}))\rightarrow (A_{\infty},\phi_{\infty})$ in the $C^{\infty}$-topology, (2) if $\|\phi_{i}\|_{L^{2}(X)}$ has no bounded sequence, then Taubes makes sense of the limit as a $\mathbb{Z}_{2}$ harmonic spinor. In particular, if the sequence $\|F_{A_{i}}\|_{L^{2}(Y)}$ is divergent to infinity, then following the inequality (\ref{E1.2}), one can see that the sequence $\|\phi_{i}\|_{L^{2}(Y)}$ is also divergent to infinity. 
\end{remark}
As a particular case of Theorem \ref{T1.1}, we have an $L^{2}$-bound on the extra fields in the fibre direction at a connection $A_{0}$. Namely, 
\begin{corollary}\label{C2}
	Let $Y$ be a closed, oriented, smooth Riemannian  three-manifold, $P$ be a principal  $SU(2)$ or $SO(3)$-bundle over $Y$. Suppose that all flat connections on the principal bundle $P$ are $non$-$degenerate$. Then for any sequence of solutions $\{(A_{0},\phi_{i})\}_{i\in\mathbb{N}}$ of Equations (\ref{I1}), there exists a subsequence $\Xi\subset\mathbb{N}$ and a positive constant $C>0$ such that $\int_{X}|\phi_{i}|^{2}\leq C$ for all $i\subset\Xi$. 
\end{corollary}
The Corollary \ref{C2} is similar to the  Vafa-Witten equations case, see\cite[Corllary 1.4]{Tanaka2017}. 

Following the notation of \cite[ Section 4.2.1]{Donaldson/Kronheimer}, we denote by $([A,\phi])$ the equivalence class of a pair $(A,\phi)$, that is a point in $\mathcal{M}(P,g)$. We denote 
$$\|(A_{1},\phi_{1})-(A_{2},\phi_{2})\|^{2}:=\|A_{1}-A_{2}\|^{2}_{L^{2}_{1}(X)}+\|\phi_{1}-\phi_{2}\|^{2}_{L^{2}_{1}(X)}.$$
We can define a distance function on $\mathcal{M}(P,g)$ as follows:
$$dist\big{(}(A_{1},\phi_{1}), (A_{2},\phi_{2})\big{)}:=\inf_{u\in\mathcal{G}_{P}}\|(A_{1},\phi_{1})-u^{\ast}(A_{2},\phi_{2})\|.$$
We can use the compactness theorem \ref{T1.1} to study the topological of the moduli space of stable flat $SL(2,\mathbb{C})$ connections.
\begin{theorem}\label{T1.3}
	Assume the hypotheses of Theorem \ref{T1.1}. Suppose that all flat connections on the principal $G$-bundle $P$ are $non$-$degenerate$. If $(A,\phi)$ is a $C^{\infty}$-solution of Equations (\ref{I1}), then there is a positive constant $\tilde{C}=\tilde{C}(P,g)\in(0,1]$ such that
	$$dist(A,M(P,g)):=\inf_{u\in\mathcal{G}_{P},\Ga\in M(P,g)}\|u^{\ast}(A)-\Ga\|_{L^{2}_{1}(Y)}\geq\tilde{C}$$
	unless $A$ is a flat connection.
\end{theorem}
\begin{remark}
	There are many combinations of conditions on $G,P,Y,g$ which imply that the flat connection is \textit{non-degenerate}. For example, if $Y$ is a closed, oriented Riemannian three-manifold with the homology of $S^{3}$, $P$ is a principal $SU(2)$-bundle over $Y$, then every flat connection $\Ga$ on $P$ is \textit{non-degenerate}.
\end{remark}
\begin{corollary}\label{C1.1}
	Let $Y$ be a closed, oriented, smooth Riemannian three-manifold with the homology of $S^{3}$, $P$ be a principal $SU(2)$-bundle. Let $\{(A_{i},\phi_{i})\}_{i\in\mathbb{N}}$ be a sequence of $C^{\infty}$-solutions of Equations (\ref{I1}). If the $L^{2}$-norms of the curvatures $\|F_{A_{i}}\|_{L^{2}(Y)}$ have a bound, then there is a subsequence of $\Xi\subset\mathbb{N}$ and a sequence of gauge transformations $\{u_{i}\}_{i\in\Xi}$ such that $\{(u^{\ast}_{i}(A_{i}),u^{\ast}_{i}(\phi_{i}))\}_{i\in\Xi}$ converges to a pair $(A_{\infty},\phi_{\infty})$ obeying Equations (\ref{I1}) on $P$ in the $C^{\infty}$-topology. Furthermore, there is a positive constant $\tilde{C}=\tilde{C}(P,g)\in(0,1]$ such that
	$$dist(A,M(P,g))\geq\tilde{C}$$
	unless $A$ is a flat connection.
\end{corollary}
The organization of this paper is as follows. In section 2, we first recall  the compactness theorem of Vafa-Witten equations which is proved by Tanaka \cite{Tanaka2017}. We also observe that the set of stable flat $SL(2,\mathbb{C})$ connections on a compact $3$-fold $Y$ is in one-to-one correspondence with solutions to a  $S^{1}$-invariant Vafa-Witten equations on $Y\times S^{1}$. Then by Tanaka's compactness theorem, we can prove a compactness theorem for stable flat $SL(2,\mathbb{C})$ connections. In section 3, we obtain a topological property of the moduli space of stable flat $SL(2,\mathbb{C})$ connections by  our compactness theorem.
\section{Compactness theorem for stable $SL(2,\mathbb{C})$ flat connections} 

\subsection{Vafa-Witten equations and stable $SL(2,\mathbb{C})$ flat connections}
In this section, we recall the compactness theorem of Vafa-Witten equations which is proved by Tanaka \cite{Tanaka2017}.  For an oriented $4$-dimensional Riemannian manifold $X$ with metric $g$, the Hodge star operator $\ast:\Om^{2}(X)\rightarrow\Om^{2}(X)$ induces the following splitting:
$$\Om^{2}(X)=\Om^{2,+}(X)\oplus\Om^{2,-}(X).$$
Accordingly, the space of $\mathfrak{g}_{P}$-valued two-forms $\Om^{2}(X,\mathfrak{g}_{P})$ splits as
$$\Om^{2,+}(X,\mathfrak{g}_{P})=\Om^{2,+}(X,\mathfrak{g}_{P})\oplus\Om^{2,-}(X,\mathfrak{g}_{P}).$$
At first, we begin to define the Vafa-Witten equations \cite{VW}. One also can see  Equations (2.4)--(2.5) in \cite{Tanaka2017}. We call the pair $(A,B)\in\mathcal{A}_{P}\times\Om^{2,+}(X,\mathfrak{g}_{P})$ a solution of  Vafa-Witten equations, if $(A,B)$ satisfies
$$d^{+,\ast}_{A}B=0,\ F^{+}_{A}+\frac{1}{8}[B.B]=0,$$ 
where $[B.B]\in\Om^{2,+}(X,\mathfrak{g}_{P})$ is defined in \cite[Appendix A]{BM}. Vafa-Witten equations were introduced by Vafa and Witten to study S-duality in twist of $\mathcal{N}=4$ supersymmetric Yang-Mills theory \cite{VW}. By appropriately counting the number of points of the moduli space of Vafa-Witten equations, we hope to obtain a number $VW(P)$ called the Vafa-Witten invariant for the principal bundle $P\rightarrow X$ \cite[Section 1.3]{BM}. These equations were also considered by Haydys \cite{Haydys} and Witten \cite{Witten} from a different point of view. 

Let $G$ be a Lie group and $P$ be a principal $G$-bundle over a smooth Riemannian manifold $X$. We recall the equivalent characterizations of flat bundles \cite[Section 1.2]{Kobayashi}, that is, bundles admitting a flat connection.
\begin{proposition}(\cite[Proposition 1.2.6]{Kobayashi})
	For a principal $G$-bundle $P$ over $X$, the following three conditions are equivalent:\\
	(1) $P$ admits a flat structure,\\
	(2) $P$ admits a flat connection,\\
	(3) $P$ is defined by a representation $\rho:\pi_{1}(X)\rightarrow G$. 
\end{proposition} 
Recall that, $\mathcal{F}_{\mathcal{A}}=F_{A}-\frac{1}{2}[\phi\wedge\phi]+{\rm{i}}d_{A}\phi$ is curvature of  the complex connection $\mathcal{A}:=A+\rm{i}\phi$. The solutions of stable flat  connections also satisfy the complex Yang-Mills equations \cite{Gagliardo/Uhlenbeck:2012}. We then have
\begin{proposition}
	If $(A,\phi)$ is a $C^{\infty}$-solution of (\ref{I1}) over a closed $n$-manifold $X$, then
	$$\int_{X}|\na_{A}\phi|^{2}+\int_{X}\langle  Ric\circ\phi,\phi\rangle+2\int_{X}|F_{A}|^{2}=0.$$
\end{proposition}
Now return to the setting of this article. Let $Y$ be an oriented, smooth, Riemannian three-manifold, $P$ be a $G$-principal bundle over $Y$ with $G$ being a compact Lie group. We denote by $X:=Y\times S^{1}$ the product manifold with the product metric. We pull back a connection $A$ on $P\rightarrow Y$ to $p_{1}^{\ast}(P)\rightarrow X$ via the canonical projection 
$$p_{1}: Y\times S^{1}\rightarrow Y.$$
We define a section $B\in\Ga(X,\Om^{2,+}\otimes\mathfrak{g}_{p^{\ast}_{1}(P)})$ as follows
$$B=(1+\ast^{X})\ast^{Y}p_{1}^{\ast}(\phi),$$ 
where $\ast^{Y}$ (resp. $\ast^{X}$) is the Hodge star operator with respect to metric $g_{Y}$ (resp. $g_{X}$).
We then have
\begin{proposition}
	The canonical projection gives a one-to-one correspondence between stable flat  connections on $P$ and $S^{1}$-invariant Vafa-Witten equations on the pullback bundle $p^{\ast}_{1}(P)$.
\end{proposition}
\begin{proof}
	The proof is similar to \cite[Lemma 8.2.2]{BM}. In a local coordinate $\{e^{1},e^{2},e^{3}\}$ of $T^{\ast}Y$, we can denote $$\phi=\phi_{1}e^{1}+\phi_{2}e^{3}+\phi_{3}e^{3}.$$
	Note that
	$$\frac{1}{2}[\phi\wedge\phi]=[\phi_{2},\phi_{3}]e^{23}+[\phi_{3},\phi_{1}]e^{31}+[\phi_{1},\phi_{2}]e^{12}.$$
	Then by the definition of $B$, we get $$B=\phi_{1}(e^{01}+e^{23})+\phi_{2}(e^{02}+e^{31})+\phi_{3}(e^{03}+e^{12}),$$ and $$-\frac{1}{4}[B.B]=[\phi_{2},\phi_{3}](e^{01}+e^{23})+[\phi_{3},\phi_{1}](e^{02}+e^{31})+[\phi_{1},\phi_{2}](e^{0,3}+e^{12}).$$
	Thus $$\frac{1}{4}[B.B]=-\frac{1}{2}(1+\ast^{X})[\phi\wedge\phi].$$ 
	We also observe that
	$$2F^{+}_{A}=(1+\ast^{X})F_{A}.$$
	Therefore, we have
	$$F^{+}_{A}+\frac{1}{8}[B.B]=(1+\ast^{X})(F_{A}-\frac{1}{2}[\phi\wedge\phi])=0.$$
	We also have an other equation $$d^{+,\ast}_{A}B=\ast^{Y}d_{A}\phi-e^{0}\wedge d_{A}^{\ast}\phi=0.$$
\end{proof}
\subsection{Compactness theorem for Vafa-Witten equations}
Mares studied the analytic aspects of Vafa-Witten equations in \cite{BM}. We don't have an $L^{2}$-bounded on the curvature $F_{A}$ of a connection $A$ which satisfies the Vafa-Witten equations as in the case of Hitchin-Simpson's equations \cite{Hitchin}. Mares observed that if $(A,B)$ is a solution of Vafa-Witten equations and the $L^{2}$-norm of $B$ has a uniform bound, then the curvature $F_{A}$ also has a uniform bound in $L^{2}$-norm by the following identity
\begin{equation*}
\begin{split}
\frac{1}{2}\|d_{A}^{+,\ast}B\|^{2}_{L^{2}(X)}&+\|F_{A}^{+}+\frac{1}{8}[B.B]\|^{2}_{L^{2}(X)}=\int_{X}(\frac{1}{12}s|B|^{2}-\frac{1}{2}W^{+}\langle B\odot B \rangle)\\
&+\frac{1}{2}\|F_{A}\|^{2}_{L^{2}(X)}+\frac{1}{4}\|\na_{A}B\|^{2}_{L^{2}(X)}+\frac{1}{16}\|[B.B]\|^{2}_{L^{2}(X)}+\frac{1}{2}\int_{X}tr(F_{A}\wedge F_{A}),\\
\end{split}
\end{equation*}
where $\odot$ denotes some bilinear on $\Om^{2,+}(X,\mathfrak{g}_{P})\otimes\Om^{2,+}(X,\mathfrak{g}_{P})$, $s$ is the scalar curvature of the metric, and $W^{+}$ is the self-dual part of the Weyl curvature of the metric (see \cite[Page 1204]{Tanaka2017} or \cite[Section B.4]{BM} for more details). Following Uhlenbeck compactness theorem, he obtained a compactness theorem of Vafa-Witten equations under the extra fields $B$ have a bound in $L^{2}$-norm \cite{BM}.

For a sequence of connections $\{A_{i}\}$ on $P$, Tanaka defined a set $S(\{A_{i}\})$ as follows:
\begin{equation}\label{E1}
S(\{A_{i}\}):=\bigcap_{r\in(0,\rho)}\{x\in X:\liminf_{i\rightarrow\infty}\int_{B_{r}(x)}|F_{A_{i}}|^{2}\geq\varepsilon_{0}\},
\end{equation}
where $\varepsilon_{0}$ is a positive constant which is defined as in \cite{Tanaka2017}. This set $S(\{A_{i}\})$ describes the singular set of a sequence of connections $\{A_{i}\}$. In \cite{Tanaka2017}, Tanaka observed that under the particular circumstance where the connections are non-concentrating and the limiting connection is \textit{non-locally reducible}, one obtains an $L^{2}$-bound on the extra fields. Here, we say that a connection $A$ on a principal $SU(2)$ or $SO(3)$ bundle $P$ is \textit{locally reducible} if the vector bundle $\mathfrak{g}_{P}$ has a one-dimensional subbundle that is $A$-covariantly constant, See \cite[Definition 2.1]{Tanaka2017}. Note that a connection on a principal $SU(2)$ or $SO(3)$ bundle $P$ being locally reducible is the same as being honestly reducible if $X$ is simply connected.  The following is an analogue of the second part of \cite[Theorem 1.1]{Taubes2017}, but under the assumption that $S(\{A_{i}\})$ is empty.
\begin{theorem}(\cite[Theorem 1.2]{Taubes2017}  and \cite[Proposition 4.4]{Tanaka2017} )\label{T4.6}
	Let $\{(A_{i},B_{i})\}$ be a sequence solutions of Vafa-Witten equations, set $r_{i}:=\|\phi_{i}\|_{L^{2}(X)}$. Let $\delta$ denote the injectivity radius of $X$. Suppose that there exist $r\in(0,\delta)$ and a sequence $\Xi\subset\mathbb{N}$ such that  
	$$\int_{B_{r}(x)}|F_{A_{i}}|^{2}<\kappa^{-2}$$
	for every $i\in\Xi$ and $x\in X$. Assume that the sequence $\{r_{i}\}_{i\in\mathbb{N}}$ has no bounded subsequence. Then there exist  a closed, nowhere dense set $Z\subset X$, a real line bundle $\mathcal{I}\rightarrow X-Z$, a section $\nu\in\Ga(X-Z,\mathcal{I}\otimes\Om^{2,+})$, a connection $A_{\De}$ on $P\mid_{X-Z}$, and an isometric bundle homomorphism $\sigma_{\De}:\mathcal{I}\rightarrow\mathfrak{g}_{P}$. Their properties are listed below:\\
	(a) $Z$ is the zero locus of $|\nu|$,\\
	(b) The function $|\nu|$ is H\"{o}lder continuous $C^{0,1/\kappa}$ on $X$,\\
	(c) The section $v$ is harmonic in the sense of $d\nu=0$,\\
	(d) $|\na\nu|$ is an $L^{2}$-function on $X-Z$ that extends as an $L^{2}$-function on $X$,\\
	(e) The curvature tensor of $A_{\De}$ is anti-self-dual,\\
	(f) The homomorphism $\sigma_{\De}$ is $A_{\De}$-covariantly constant.\\
	In addition, there exist a subsequence $\Lambda\subset\Xi$ and a sequence $u_{i}$ of automorphisms from $P$ such that\\
	(i) $\{u_{i}^{\ast}(A_{i})\}$ converges to $A_{\De}$ in the $L^{2}_{1}$ topology on compact subset in $X-Z$ and\\
	(ii) The sequence $\{r^{-1}_{i}u^{\ast}_{i}(B_{i})\}$ converges to $v\otimes\sigma_{\De}$ in the $L^{2}_{1}$-topology  on compact subset in $X-Z$ and the $C^{0}$-topology on $X$. Meanwhile, $\{r_{i}^{-1}|B_{i}|\}$ converges to $|\nu|$ in the weakly $L^{2}_{1}$-topology and the $C^{0}$-topology on the whole of $X$.
\end{theorem}
\subsection{Proof of our results}
In this section, we give the proof of our main result. At first, we observe that
\begin{proposition}
	Let $Y$ be a closed, oriented, smooth Riemannian three-manifold, $P$ be a principal $G$-bundle with $G$ being a compact Lie group. Let $\{A_{i}\}_{i\in\mathbb{N}}$ be a sequence $C^{\infty}$-connections on $P$ with the $L^{2}$-norms of the curvatures $\|F_{A_{i}}\|_{L^{2}(X)}$ have a uniform bound. We denote by $\{\textbf{A}_{i}\}$ the pullback $S^{1}$-invariant connections. Then the set $S(\{\textbf{A}_{i}\})$ is empty, where $S(\cdot)$ is defined in Equation (\ref{E1}).
\end{proposition}
\begin{proof}
	For a point $(y_{0},\theta_{0})\in Y\times S^{1}$, we denote by  \begin{equation}\nonumber
	\begin{split}
	B_{r}(y_{0},\theta_{0})&:=\{(y,\theta):|y-y_{0}|_{g_{Y}}^{2}+|\theta-\theta_{0}|^{2}<r^{2}\}\\
	&\subset (-r+\theta_{0},r+\theta_{0})\times B_{r}(y_{0})\\
	\end{split}
	\end{equation}
	the geodesic ball on $Y\times S^{1}$. Hence, we have
	\begin{equation}\nonumber
	\begin{split}
	\|F_{\textbf{A}_{i}}\|^{2}_{L^{2}(B_{r}(y_{0},\theta_{0}))}&=\int_{B_{r}(y_{0},\theta_{0})}|F_{\textbf{A}_{i}}|^{2}dvol_{g_{Y}}d\theta\\
	&\leq\int_{-r+\theta_{0}}^{r+\theta_{0}}d\theta\int_{B_{r}(y_{0})}|F_{A_{i}}|^{2}\\
	&\leq2r\sup_{i}\|F_{A_{i}}\|^{2}_{L^{2}(Y)}.\\
	\end{split}
	\end{equation}
	We can choose $r$ sufficiently small such that $$2r\sup\|F_{A_{i}}\|_{L^{2}(Y)}<\kappa^{-2},$$
	where $\kappa$ is the constant on Theorem \ref{T4.6}. Therefore, we complete this proof.
\end{proof}
Following the idea in the proof of \cite[Theorem 1.2]{Taubes2014}, we can obtain a compactness theorem for the stable flat $SL(2,\C)$ connections on three-manifold. 
\begin{theorem}\label{T1}
	Let $Y$ be a closed, oriented, smooth Riemannian three-manifold, $P$ be a principal $G$-bundle with $G$ being $SU(2)$ or $SO(3)$. Let $\{(A_{i},\phi_{i})\}_{i\in\mathbb{N}}$ be a sequence of $C^{\infty}$-solutions of Equations (\ref{I1}), set $r_{i}:=\|\phi_{i}\|_{L^{2}(X)}$. Suppose that the $L^{2}$-norms of the curvatures $\|F_{A_{i}}\|_{L^{2}(Y)}$ have a uniform bound and the sequence $\{r_{i}\}_{i\in\mathbb{N}}$ has no bounded subsequence. Then there exist  a closed, nowhere dense set $Z_{Y}\subset Y$, a real line bundle $\mathcal{I}_{Y}\rightarrow Y-Z_{Y}$, a section $\nu\in\Ga(Y-Z_{Y},\mathcal{I}_{Y}\otimes\Om^{1})$, a connection $A_{\De}$ on $P\mid_{Y-Z_{Y}}$, and  an isometric bundle homomorphism $\sigma_{\De}:\mathcal{I}_{Y}\rightarrow\mathfrak{g}_{P}$. Their properties are listed below:\\
	(a) $Z_{Y}$ is the zero locus of $|\nu|$,\\
	(b) The section $v$ is harmonic in the sense of $d\nu=d^{\ast}\nu=0$,\\
	(c) The curvature tensor of $A_{\De}$ is flat,\\
	(d) The homomorphism $\sigma_{\De}$ is $A_{\De}$-covariantly constant.\\
	In addition, there exist a subsequence $\Lambda\subset\Xi$ and a sequence $u_{i}$ of automorphisms from $P$ such that\\
	(i) $\{u_{i}^{\ast}(A_{i})\}$ converges to $A_{\De}$ in the $L^{2}_{1}$ topology on compact subset in $Y-Z_{Y}$ and\\
	(ii) The sequence $\{r^{-1}_{i}u^{\ast}_{i}(\phi_{i})\}$ converges to $\nu\otimes\sigma_{\De}$ in the $L^{2}_{1}$ topology on compact subset in $Y-Z_{Y}$ and the $C^{0}$-topology on $Y$. Meanwhile, $\{r_{i}^{-1}|\phi_{i}|\}$ converges to $|\nu|$ in the weakly $L^{2}_{1}$-topology and the $C^{0}$-topology on the whole of $Y$.
\end{theorem}
\begin{proof}
	As explained momentarily, this theorem constitutes a special case to Theorem \ref{T4.6}. To obtain Theorem \ref{T1} from Theorem \ref{T4.6}, we take $X$ in Theorem \ref{T1} to be the product $Y\times S^{1}$ with the metric being the product metric. The pull-back of the principal $G$-bundle $P$ on $Y$ to $X$ via the projection map to $Y$ defines a principal $G$-bundle over $X$, the latter is denoted also by $P$. Let $\{ (A_{i},\phi_{i})\}$ be a sequence solutions of stable flat $SL(2,\mathbb{C})$ connections over $Y$. For simplicity we keep the same notations for objects on $Y$ and their pullbacks to $X$. We denote $B_{i}=(1+\ast^{X})\ast^{Y}\phi_{i}$. If suppose that the sequence $\{r_{i}:=\|\phi_{i}\|_{L^{2}(Y)}\}_{i\in\mathbb{N}}$ has no bounded subsequence, then $\|B_{i}\|_{L^{2}(X)}$ also has no bounded subsequence. A similar sort of argument can be used to prove that Theorem \ref{T4.6}'s set $Z$ is the product of $S^{1}$ and a closed set $Z_{Y}\subset Y$ and that Theorem \ref{T4.6}'s real line bundle $\mathcal{I}$ is isomorphic to the pull-back via the projection map of a real line bundle defined on the complement in $Y$ of $Z_{Y}$, this denoted for now by $\mathcal{I}_{Y}$. Moreover, such an isomorphism identifies Theorem \ref{T4.6}'s version of $\nu$ with the pull-back of a harmonic, $\mathcal{I}_{Y}$ valued 1-form on $Y-Z_{Y}$ with $Z_{Y}$ denoting the locus where its norm is zero.
\end{proof}
\begin{remark}
	Taubes considered a sequence of complex connections  $\{ \mathcal{A}_{i}:=A_{i}+{\rm{i}}\phi_{i}\}$ such that the $L^{2}$-norms of $\{F_{\mathcal{A}_{i}}\}$ are bounded. If $\{\|\phi_{i}\|_{L^{2}(X)}\}$ has no bounded sequence, then Taubes makes sense of the limit as a $\mathbb{Z}_{2}$ harmonic spinor \cite{Taubes2014}. In our result,  we add the conditions that all connections $\{\mathcal{A}_{i}\}$ are  stable flat $SL(2,\mathbb{C})$ connection and the real curvatures $\{F_{A_{i}} \}$ have $L^{2}$-bounded, then we prove the limit as a decoupled stable flat $SL(2,\mathbb{C})$ connection.
\end{remark}
The next theorem is a special case of Theorem 1.1b in \cite{Taubes2014}. It implies among other things that $Z_{Y}$ has measure zero. To set the notation for this upcoming theorem. A point $p\in Z_{Y}$ is a point of discontinuity for $\mathcal{I}_{Y}$, if $\mathcal{I}_{Y}$ is not isomorphic to the product bundle on the complement of $Z_{Y}$ in any neighborhood of $p$ \cite{Taubes2014}. 
\begin{theorem}(\cite[Theorem 1.1b]{Taubes2014})\label{T3}
	Let $Z_{Y}$ and $\mathcal{I}_{Y}$ be as described in Theorem \ref{T1}. The set $Z_{Y}$ has Hausdorff dimension at most 1, and moreover, the set of the points of discontinuity for $\mathcal{I}_{Y}$ (defined in the preceding paragraph) are the points in the closure of an open subset of $Z_{Y}$ that is an embedded $C^{1}$ curve in $Y$ denoted by $\Sigma$.
\end{theorem} Uhlenbeck's  \cite{Uhlenbeck1982} theorem applies to the connections on $P$ and in particular makes the following assertion:\\
\textbf{Uhlenbeck's Theorem}: Let $\{A_{i}\}_{i\in\mathbb{N}}$ be a sequence of connections on $P$ over a closed, oriented, $3$-manifold. If $L^{2}$-norms of the curvatures $F_{A_{i}}$ of the connections $\{A_{i}\}$ have a uniform bound, then there is a subsequence $\Xi\subset\mathbb{N}$ and a sequence of gauge transformations $\{u_{i}\}_{i\in\Xi}$ such that $\{u^{\ast}_{i}(A_{i})\}$ converges weakly in the $L^{2}_{1}$-topology to a connection $A_{\infty}$ on $P$. 

By the priori estimate (\ref{E1.2}), we then have
\begin{theorem}(\cite[Theoreom 1.1a]{Taubes2014})\label{T2.10}
	Let $Y$ be a closed, oriented, smooth Riemannian three-manifold, $P$ be a principal  $SU(2)$ or $SO(3)$-bundle over $Y$. Let $\{(A_{i},\phi_{i})\}_{i\in\mathbb{N}}$ be a sequence of $C^{\infty}$-solutions of Equations (\ref{I1}). Suppose that the sequence $\{\|\phi_{i}\|_{L^{2}(Y)}\}$ has a bounded subsequence. Then there is a subsequence of $\Xi\subset\mathbb{N}$ and a sequence of gauge transformations $\{u_{i}\}_{i\in\Xi}$ such that $\{(u^{\ast}_{i}(A_{i}),u_{i}^{\ast}(\phi_{i}))\}_{i\in\Xi}$ converges to a pair $(A_{\infty},\phi_{\infty})$ obeying Equations (\ref{I1}) on $P$ in $C^{\infty}$-topology. 
\end{theorem}
We are finally ready to use the above results in the following proposition.
\begin{proposition}\label{P1}
	Let $Z_{Y}$ and $\mathcal{I}_{Y}$ be as described in Theorem \ref{T1}, so that $\sigma_{\De}$ and $A_{\De}$ are defined over $Y-Z_{Y}$. Then\\
	(1) There exists a smooth flat connection $A_{\infty}$ defined over all of $Y$, and a Sobolev class $L^{2}_{2}$ gauge transformation $u_{\infty}$ defined over $Y-Z_{Y}$ such that $u_{\infty}^{\ast}(A_{\infty})$ is restriction to $Y-Z_{Y}$ of $A_{\infty}$. Defining $\sigma_{\infty}:=u_{\infty}^{\ast}(\sigma_{\De})$ over $Y-Z_{Y}$, then $\na_{A_{\infty}}\sigma_{\infty}=0$.\\
	(2) The bundle $\mathcal{I}_{Y}$ over $Y-Z_{Y}$ extends to a bundle defined over all of $Y$, which we again denote by $\mathcal{I}_{Y}$,\\
	(3) There exist extensions of both $v\in\Gamma(\mathcal{I}_{Y}\otimes\Om^{1})$ and $\sigma_{\infty}:\mathcal{I}_{Y}\rightarrow\mathfrak{g}_{P}$ to all of $Y$. We again denote these by $v$ and $\sigma_{\infty}$. The extensions satisfy $dv=0$ and $\na_{A_{\infty}}\sigma_{\infty}=0$.
\end{proposition}
\begin{proof}
	The idea of our proof is similar to \cite[Proposition 4.6]{Tanaka2017}. 
	
	We first prove item 1.  Following weak Uhlenbeck compactness theorem (see \cite[ Theorem A]{Wehrheim}), for any sequence $\{A_{i}\}_{i\in N}$ with bounded $L^{2}$-curvature $\{F_{A_{i}}\}_{i\in\N}$ on a principal $G$-bundle over a closed three-manifold, there exists a subsequence  (again denote $\{A_{i}\}_{i\in\N}$ ) and a sequence of gauge transformations $\{u_{i}\}_{i\in\N}$ such that $u^{\ast}_{i}(A_{i})$  converges weakly to a limit connection $A_{\infty}$  over all of $Y$ in $L^{2}_{1}$. Recall from Theorem \ref{T1} that $A_{\De}$ is the $L^{2}_{1}$ limit over compact subset $Y-Z_{Y}$ of gauge equivalent connections. Since weakly $L^{2}_{1}$ limits preserve $L^{2}_{2}$ gauge equivalence, it follows there exists a Sobolev-class $L^{2}_{2}$ gauge transformation $u_{\infty}$ such that $u_{\infty}^{\ast}(A_{\De})=A_{\infty}$.
	
	Note that $A_{\De}$ is flat and gauge-equivalent over the complement of $Z_{Y}$ to $A_{\infty}$. Thus, $A_{\infty}$ is a $L^{2}_{1}$ connection whose curvature is $L^{2}$, and vanishes on the complete of $Z_{Y}$, which by Theorem \ref{T3} is a set of measure zero. Hence the curvature of $A_{\infty}$ is flat and so a standard elliptic regularity argument can be used to prove that there is an $L^{2}_{2}$ and $C^{0}$ automorphism of $P$ that transforms $A_{\infty}$ into a smooth flat connection. After possibly composing $u_{\infty}$ with such an automorphism, we may assume without loss of generality that $A_{\infty}$ is smooth and that $u_{\infty}$ is continuous. That $\na_{A_{\infty}}\sigma_{\infty}=0$ follows from Theorem \ref{T1} since $\sigma_{\De}$ is $A_{\De}$-covariantly constant. This establishes the item 1.
	
	We next prove the item 2 that $\mathcal{I}_{Y}$ extends over $Z_{Y}$. Let $\Sigma\subset Z_{Y}$ denote the $C^{1}$ submanifold that is described by Theorem \ref{T3}. It is enough to prove that $\Sigma$ is empty. For this purpose, assume to the contrary that $\Sigma\neq\emptyset$ and let $S\subset\Sigma$ be a component. This is a $C^{1}$ embedded curve. Fix a point $p\in S$. Since $S$ is $C^{1}$, there is an embedded disk  $D\subset Y$ closure intersects $S$ transversally at a single point which is $p$. This is also its only intersection point with $Z_{Y}$ since $S$ is an open subset of $Z_{Y}$. Since $p$ is a point of discontinuity for the bundle $\mathcal{I}_{Y}$, the restriction of $\mathcal{I}_{Y}$ to $D-\{p\}$ is not isomorphic to the product line bundle. In particular, parallel transport by $A_{\De}$ of $\sigma_{\De}$ along any circle in $D-\{p\}$ which wraps once around $p$ gives $-\sigma_{\De}$. However, $A_{\De}$ is gauge-equivalent to a connection which is smooth over all of $D$. This parallel transport around sufficiently small bounded interval will be arbitrarily close to $+\sigma_{\De}$, which is a contradiction.
	
	Finally, we prove item 3 by showing that both $v$ and $\sigma_{\infty}$  extend to all of $Y$ as $L^{2}_{1}$ sections. Granted this extension, we may argue as in item 1 that both $dv$ and $\na_{A_{\infty}}\sigma_{\infty}$ are $L^{2}$ sections which vanish almost everywhere, and hence by elliptic regularity, $v$ and $\sigma_{\infty}$ are smooth and satisfy $dv=0$ and $\na_{A_{\infty}}\sigma_{\infty}=0$ over all of $Y$.
\end{proof}
Following above results, we can prove a Uhlenbeck-type compactness theorem on $Y$ for stable flat $SL(2,\mathbb{C})$ connections satisfying an $L^{2}$-bound for the real curvature.\\
\begin{proof}[\textbf{Proof of Theorem \ref{T1.1}}]
	We set $r_{i}:=\|\phi_{i}\|_{L^{2}(Y)}$. At first, we can prove that there exists a subsequence $\Xi\subset\N$ such that $\{r_{i}\}_{i\in\Xi}$ have a uniform bound. If not, then the sequence $\{r_{i}\}_{i\in\mathbb{N}}$ has no bounded subsequence. We denote $A_{\infty}$, $\nu$ and $\sigma_{\infty}$ as described in Proposition \ref{P1}. Hence, following Proposition \ref{P1}, we have
	$$d_{A_{\infty}}(\nu\otimes\sigma_{\infty})=d\nu\otimes\sigma_{\infty}-\nu\otimes\na_{A_{\infty}}\sigma_{\infty}=0$$ and
	$$d_{A_{\infty}}\ast(\nu\otimes\sigma_{\infty})=d\ast \nu\otimes\sigma_{\infty}+\ast\nu\otimes \na_{A_{\infty}}\sigma_{\infty}=0.$$
	Since $\nu\otimes\sigma_{\infty}\in\Om^{1}(Y,\mathfrak{g}_{P})$, the hypothesis of the flat connection $A_{\infty}$ implies that $$\nu\otimes\sigma_{\infty}=0.$$
	Following the item 1 in Proposition \ref{P1}, there exist a continuous Sobolev-class $L^{2}_{2}$ gauge transformation $u_{\infty}$ defined over $Y-Z_{Y}$ such that
	$$(u^{-1}_{\infty})^{\ast}(A_{\infty})=A_{\De}$$ and $$(u^{-1}_{\infty})^{\ast}(\sigma_{\infty})=\sigma_{\De}.$$ 
	Hence
	$$\nu\otimes\sigma_{\De}=(u_{\infty}^{-1})^{\ast}(\nu\otimes\sigma_{\infty})=0$$ on $Y-Z_{Y}$. The zero locus of the extension of $|\nu|$ is the set $Z_{Y}$. And we can set $\sigma_{\De}$ is a unit length, $A_{\De}$-covariantly constant homomorphism over $Y-Z_{Y}$. Hence, we can say $|\nu|=0$ on $Y$. 
	
	On the other hand, following the last item in Theorem \ref{T1}, there exist a subsequence $\Xi\subset\mathbb{N}$ and a sequence $\{u_{i}\}_{i\in\Xi}$ of automorphisms from $P$ such that $\{r^{-1}_{i}u^{\ast}_{i}(\phi_{i})\}$ converges to $\nu\otimes\sigma_{\De}$ in the $L^{2}_{1}$-topology on compact subset in $Y-Z_{Y}$ and the $C^{0}$-topology on $Y$. Meanwhile, $\{r_{i}^{-1}|\phi_{i}|\}$ converges to $|\nu|$ in the weakly $L^{2}_{1}$-topology and the $C^{0}$-topology on the whole of $Y$. Hence $$\lim_{i\rightarrow\infty}r^{-1}_{i}\|\phi_{i}\|_{L^{2}(Y)}=0.$$
	It's contradicting the fact $r^{-1}_{i}\|\phi_{i}\|_{L^{2}(Y)}=1$, $\forall i\in\mathbb{N}$. In particular, the preceding argument shows that there exists a subsequence $\Xi\subset\mathbb{N}$ such that $\{r_{i}\}_{i\in\Xi}$ have a uniform bound. Thus following Theorem \ref{T2.10}, there is a subsequence (again denote by $\{(A_{i},\phi_{i})\}_{i\in\Xi}$) and a sequence of gauge transformation $\{u_{i}\}_{i\in\Xi}$ such that  $\{(u^{\ast}_{i}(A_{i}),u^{\ast}_{i}(\phi_{i}))\}_{i\in\Xi}$ converges to a pair $(A_{\infty},\phi_{\infty})$ on $P$ in the $C^{\infty}$-topology.
\end{proof}
\begin{corollary}
	Let $Y$ be a closed, oriented Riemannian three-manifold with the homology of $S^{3}$, $P$ be a principal $SU(2)$-bundle. Let $\{(A_{i},\phi_{i})\}_{i\in\mathbb{N}}$ be a sequence of $C^{\infty}$ solutions of Equations (\ref{I1}). If the $L^{2}$-norms of the curvatures $\|F_{A_{i}}\|_{L^{2}(Y)}$ have a uniform bound, then there is a subsequence of $\Xi\subset\mathbb{N}$ and a sequence of gauge transformations $\{u_{i}\}_{i\in\mathbb{N}}$ such that $\{(u^{\ast}_{i}(A_{i}),u^{\ast}_{i}(\phi_{i}))\}_{i\in\Xi}$ converges to a pair $(A_{\infty},\phi_{\infty})$ obeying Equations (\ref{I1}) on $P$ in the $C^{\infty}$-topology.
\end{corollary}
\section{ Disconnectedness of the moduli space $\mathcal{M}(P,g)$}
\subsection{A lower positive bound of extra fields}
We call a stable flat connection $\mathcal{A}:=A+\rm{i}\phi$ \textit{decoupled}, if the real connection $A$ is flat and the extra field $\phi$ is a harmonic $\mathfrak{g}_{P}$-$1$-form with respect to $d_{A}+d^{\ast}_{A}$, i.e, 
$$F_{A}=0\ and\ d_{A}\phi=d^{\ast}_{A}\phi=0.$$ 
Using a result of Uhlenbeck \cite{Uhlenbeck1985}, the author observed that if the stable flat connection $\mathcal{A}$ over a closed, smooth, Riemannian three-manifold $Y$, then  the $L^{2}$-norm of extra fields has a uniform positive lower bound unless the real connection is flat.
\begin{theorem}\label{T3.7}({\cite{Huang}})
	Let $Y$ be a closed, oriented, Riemannian three-manifold and endowed with a smooth Riemannian metric $g$, $P$ be a principal $G$-bundle with $G$ being a compact Lie group. If $(A,\phi)$ is a $C^{\infty}$-solution of equations (\ref{I1}), then there is a positive constant $C=C(X,g,G)$ such that  $$\|\phi\|_{L^{2}(Y)}\geq C$$ unless $A$ is a flat connection.
\end{theorem}
Suppose that all flat connections on $P$ are \textit{non-degenerate}. Then the extra fields vanish if the stable flat connection $A+\rm{i}\phi$ is \textit{decoupled} over a closed Riemannian manifold.  Following Theorem \ref{T3.7}, we then have
\begin{corollary}\label{C1}
	Assume the hypothesis of Theorem \ref{T3.7}. Suppose that all flat connections on the principal bundle $P$ are $non$-$degenerate$.  If $(A,\phi)$ is a $C^{\infty}$-solution of equations (\ref{I1}), then either there exists a positive constant $C=C(X,g,G)$ such that $$\|\phi\|_{L^{2}(Y)}\geq C$$ or $\phi$ vanishes.
\end{corollary}
\subsection{A lower positive bound of curvatures}
One can see that  $M(P,g)$ is the space of real flat connections and  $\mathcal{M}'(P,g):=\mathcal{M}(P,g)\backslash M(P,g)$ is the space of real connections are non-flat. Hence we can denote by 
\begin{equation}\nonumber
\begin{split}
dist(M(P,g),\mathcal{M}'(P,g)):&=\inf_{u\in\mathcal{G}_{P},\Ga\in M(P,g)}\|u^{\ast}(A,\phi)-(\Ga,0)\|\\
&=\big{(}\inf_{u\in\mathcal{G}_{P},\Ga\in M(P,g)}\|u^{\ast}(A)-\Ga\|^{2}_{L^{2}_{1}(X)}+\|\phi\|^{2}_{L^{2}_{1}(X)}\big{)}^{\frac{1}{2}}\\
\end{split}
\end{equation}
the distance between $M(P,g)$ and $\mathcal{M}'(P,g)$. Following Theorem \ref{T3.7}.  We can obtain a topological property of the moduli space $\mathcal{M}(P,g)$.
\begin{proposition}( Disconnectedness of the moduli space $\mathcal{M}(P,g)$). Let $Y$ be a closed, oriented, smooth, Riemannian three-manifold, $P$ be a principal $G$-bundle with $G$ being a compact Lie group. Suppose that all flat connections on the principal bundle $P$ are $non$-$degenerate$. If the moduli spaces $M(P,g)$  and $\mathcal{M}'(P,g):=\mathcal{M}(P,g)\backslash M(P,g)$ are all non-empty, then the moduli space $\mathcal{M}(P,g)$ is disconnected.
\end{proposition}
\begin{proof}
	Under the hypothesis of the flat connection, following the Corollary \ref{C1}, the $L^{2}$-norm of the extra field $\|\phi\|_{L^{2}(X)}$ has a lower bound unless $\phi$ vanishes. If the moduli spaces $M(P,g)$  and $\mathcal{M}(P,g)\backslash M(P,g)$ are all non-empty, then $$dist(M(P,g),\mathcal{M}'(P,g)\geq\|\phi\|_{L^{2}(Y)}\geq C,$$ 
	where $C=C(Y,g)$ is the positive constant in Corollary \ref{C1}, i.e., the moduli space $\mathcal{M}(P,g)$ is disconnected.
\end{proof}
We extend the idea in \cite{Huang3} to stable flat $SL(2,\mathbb{C})$ connection case, we prove a gap result of the real curvature following the compactness theorem \ref{T1.1}. 
\begin{proposition}\label{P4.3}
	Let $Y$ be a closed, oriented, smooth, Riemannian three-manifold, $P$ be a principal $G$-bundle with $G$ being $SU(2)$ or $SO(3)$. Suppose that all flat connections on the principal bundle $P$ are $non$-$degenerate$ . If the pair $(A,\phi)$ is a $C^{\infty}$-solution of equations (\ref{I1}), then there is a positive constant $C=C(Y,g,P)$  such that $$\|F_{A}\|_{L^{2}(Y)}\geq C$$ unless the real connection $A$ is flat.
\end{proposition}
\begin{proof}
	Suppose that the constant $C$ does not exist. We may then choose a sequence $\{(A_{i},\phi_{i})\}_{i\in\N}$ such that $\|F_{A_{i}}\|_{L^{2}(Y)}\rightarrow0\ as\ i\rightarrow\infty$,
	and $\{A_{i}\}_{i\in\N}$ are all non-flat. Thus the compactness Theorem \ref{T1.1} implies that there exists a pair $(A_{\infty},\phi_{\infty})$ obeys the equations (\ref{I1}) and there is a sequence of gauge transformations $\{u_{i}\}_{i\in\Xi}$ such that
	$(u_{i}^{\ast}(A_{i}),u_{i}^{\ast}(\phi_{i}))\rightarrow (A_{\infty},\phi_{\infty})$ in $C^{\infty}$ over  $Y$. Following  Theorem \ref{T3.7}, the $L^{2}$-norm of extra field $\|\phi_{\infty}\|_{L^{2}X}$ has a positive lower bound. Therefore, we have
	$$\|\phi_{\infty}\|^{2}_{L^{2}(Y)}=\lim_{i\rightarrow\infty}\int_{Y}|\phi_{i}|^{2}\geq \title{C},$$ 
	where $\title{C}=\title{C}(Y,g)$ is a positive constant. 
	
	On the other hand, since $\|F_{A_{i}}\|_{L^{2}(Y)}\rightarrow0$, the weak Uhlenbeck compactness theorem implies that the connection $A_{\infty}$ on $P$ is flat. Hence $A_{\infty}$ is \textit{non-degenerate} by the hypothesis on this proposition. It implies that the extra field $\phi_{\infty}\equiv0$. It's contradiction with $\|\phi_{\infty}\|_{L^{2}(Y)}$ has a uniform positive lower bound. 
\end{proof}
\begin{remark}
	The solutions of stable flat connections also satisfy the complex Yang-Mills equations \cite{Gagliardo/Uhlenbeck:2012}. The author proved that if the pair $(A,\phi)$ is a smooth solution of stable flat connection over a closed, smooth, Riemannian $n$-manifold $X$, the curvature $F_{A}$ of non-flat connection $A$ has a uniform positive lower $L^{p}$-bound under the condition that all flat connections are all \textit{non-degenerate}, See \cite[Theorem 1.2]{Huang2}. 
\end{remark}
\begin{proof}[\textbf{Proof of Theorem \ref{T1.3}}]
	At first, we give the priori estimate for the curvature of connection. Since $$F_{A}=F_{\Ga}+d_{\Ga}a+a\wedge a,$$
	we have
	\begin{equation}\nonumber
	\begin{split}
	\|F_{A}\|^{2}_{L^{2}(Y)}&\leq\|d_{\Ga}a\|^{2}_{L^{2}(Y)}+\|a\wedge a\|^{4}_{L^{2}(Y)}\\
	&\leq\|d_{\Ga}a\|^{2}_{L^{2}(Y)}+\|a\|^{4}_{L^{4}(Y)}\\
	&\leq\|d_{\Ga}a\|^{2}_{L^{2}(Y)}+C_{S}\|a\|^{4}_{L^{2}_{1}(Y)}.\\
	\end{split}
	\end{equation}
	The last inequality, we use the Sobolev embedding $L^{2}_{1}\hookrightarrow L^{p}$ for $2\leq p\leq6$ with embedding constant $C_{S}$. Here $$|d_{\Ga}a|\leq2|\na_{\Ga}a|$$ due to the fact that  $$2d_{\Ga}a(U,V)=\na_{\Ga}a(U,V)-\na_{\Ga}a(U,V)$$ for all $U,V\in T_{y}(Y)$. Combining the preceding inequalities yields
	$$\|F_{A}\|^{2}_{L^{2}(Y)}\leq c(\|a\|^{2}_{L^{2}_{1}(Y)}+\|a\|^{4}_{L^{2}_{1}(Y)}).$$
	where $c=c(Y,g)$ is a positive constant. 
	
	If $\|a\|_{L^{2}_{1}(Y)}\leq 1$, then
	$$\|F_{A}\|^{2}_{L^{2}(Y)}\leq 2c\|a\|^{2}_{L^{2}_{1}(Y)}.$$ 
	Therefore, we have $$\|a\|^{2}_{L^{2}_{1}(Y)}\geq\frac{C}{2c},$$
	where $C$ is the positive constant in Proposition \ref{P4.3}.  We set $\tilde{C}:=\min\{1,\frac{C}{2c}\}$, thus $$\inf_{u\in\mathcal{G}_{P},\Gamma\in M(P,g)}\|u^{\ast}(A)-\Gamma\|^{2}_{L^{2}_{1}(Y)}\geq\tilde{C}.$$
	We complete this proof.
\end{proof}

\section*{Acknowledgements}
We would like to thank the anonymous referees for careful reading of my manuscript and helpful comments. I would like to thank Y. Tanaka for kind comments regarding this and it companion article \cite{Tanaka2017}. This work was supported in part by NSF of China (11801539) and the Fundamental Research Funds of the Central Universities (WK3470000019), the USTC Research Funds of the Double First-Class Initiative (YD3470002002).

\end{document}